\documentclass[11pt]{article}
\usepackage[dvips]{epsfig}
\usepackage{amsmath,amssymb,euscript,graphicx,amsfonts,verbatim}
\usepackage{enumerate,color}
\usepackage{graphicx}

\usepackage{graphicx}

\usepackage[colorlinks=true,linkcolor=blue,urlcolor=blue,citecolor=blue]{hyperref}

\usepackage{enumitem}

\newtheorem{theorem}{Theorem}[section]
\newtheorem{proposition}[theorem]{Proposition}
\newtheorem{lemma}[theorem]{Lemma}

\newtheorem{definition}[theorem]{Definition}

\newtheorem{notation}[theorem]{Notation}

\newenvironment{proof}{{\noindent \sc Proof.}}{\hfill $\Qed$\\}

\newcommand{\la}{\langle}
\newcommand{\ra}{\rangle}

\def\GL{\hbox{\rm GL}}

\newcommand{\Qed}{\rule{2.5mm}{3mm}}

\newcommand{\Cay}{\hbox{{\rm Cay}}}

\newcommand{\Ga}{\Gamma}
\newcommand{\bfv}{\mbox{\boldmath $v$}}

\newcommand{\irr}[1]{\operatorname{Irr}(#1)}

\newcommand{\ZZ}{\mathbb{Z}}

\newcommand{\CC}{\mathbb{C}}

\newcommand{\mb}{\mathbf}

\newcounter{case}

\renewcommand{\thecase}{\arabic{case}}

\newcounter{subcase}

\numberwithin{subcase}{case}

\begin{document}


\begin{center}
{\bf\Large On closed distance magic circulants of valency up to $5$} \\ [+4ex]
Blas Fernández{\small$^{a,b}$},  
Roghayeh Maleki{\small$^{a,b,}$\footnote{Corresponding author e-mail:~roghayeh.maleki@famnit.upr.si}},  
\v Stefko Miklavi\v c{\small$^{a, b, c}$}
\\
Andriaherimanana Sarobidy Razafimahatratra{\small$^{a, b}$} 
\\ [+2ex]
{\it \small 
$^a$University of Primorska, UP IAM, Muzejski trg 2, 6000 Koper, Slovenia\\
$^b$University of Primorska, UP FAMNIT, Glagolja\v ska 8, 6000 Koper, Slovenia\\
$^c$IMFM, Jadranska 19, 1000 Ljubljana, Slovenia}
\end{center}


\begin{abstract}
Let $\Ga=(V,E)$ be a graph of order $n$. A {\em closed distance magic labeling} of  $\Ga$ is a bijection $\ell : V \to \{1,2, \ldots, n\}$ for which there exists a positive integer $r$ such that $\sum_{x \in N[u]} \ell(x) = r$ for all vertices $u \in V$, where $N[u]$  is the closed neighborhood of $u$. A graph is said to be {\em closed distance magic} if it admits a closed distance magic labeling. 

In this paper, we classify all connected closed distance magic circulants with valency at most $5$, that is, Cayley graphs $\Cay(\ZZ_n;S)$ where $|S| \le 5$ and $S$ generates $\ZZ_n$. 
\end{abstract}

\begin{quotation}
\end{quotation}

\begin{quotation}
\noindent 
{\em Math. Subj. Class.:} 05C78, 05C25, 05C50.
\end{quotation}


\section{Introduction}
\label{sec:intro}
\noindent

All graphs considered in this paper are finite, simple and undirected.  A {\em distance magic labeling} of a graph $\Gamma$ of order $n$ is a bijective labeling of vertices of $\Gamma$ with positive integers $1, 2, \ldots, n$, such that the sum of the labels of the neighbors of a vertex does not depend on a given vertex. In such a case this sum is called the {\em magic constant} of the graph in question and the graph itself is said to be {\em distance magic}. The survey~\cite{AruFroKam11} gathers most of the results on distance magic graphs prior to 2010 (but see also~\cite{Gal_sur} for some of the more recent results). It is well known that the valency of a regular distance magic graph must be even. The obvious fact that the only distance magic cycle is the $4$-cycle thus led Rao~\cite{Rao08} to propose the problem of characterizing all tetravalent distance magic graphs. 

Cayley graphs (over certain group $G$) appear as a natural examples of regular graphs. Cayley graphs were extensively studied, as they enable to encode the abstract structure of a group. If group $G$ is cyclic, then a Cayley graph over $G$ is called a {\em circulant}. 

Tetravalent distance magic circulant graphs were first studied in \cite{CicFro16}, where a partial classification of these graphs was given. A complete classification of these graphs was later given in \cite{miklavivc2021classification}. Miklavi\v{c} and \v{S}parl also studied distance magic circulant graphs with valency $6$, and obtained a partial classification of these graphs, see \cite{MS6}.

A related concept of distance magic graphs is the one of {\em closed distance magic} graphs.  A {\em closed distance magic labeling} of a graph $\Gamma$ of order $n$ is a bijective labeling of vertices of $\Gamma$ with positive integers $1, 2, \ldots, n$, such that for all vertices in the graph, the sum of the labels of the neighbors of a fixed vertex including  the label of the vertex itself is independent of the choice of the given vertex. In such a case this sum is called the {\em closed magic constant} of the graph in question and the graph itself is said to be {\em closed distance magic}. The investigation of closed distance magic circulant graphs was initiated in \cite{Simanjuntak2013MagicLO} and later continued in \cite{anholcer2016spectra}, in which results about closed distance magic circulant graphs with specific connection sets $S$ were proven. 

In this paper we study connected closed distance magic circulant graphs with valency at most $5$. It is easy to see that the only connected closed distance magic circulant graphs with valency $1$ or $2$ are the complete graphs $K_2$ and $K_3$, respectively. Therefore, we concentrate on circulant graphs with valencies $3$, $4$, and $5$. The main results of this paper are the following two theorems.

\begin{theorem}
	\label{thm:val3,4}
	Let $\Gamma$ denote a connected circulant graph with valency $3$ or $4$. Then, $\Gamma$ is closed distance magic if and only if $\Gamma$ is isomorphic either the complete graph $K_4$ or to the complete graph $K_5$.
\end{theorem}

\begin{theorem}
		\label{thm:val5}
	Let $\Gamma$ denote a connected circulant graph with valency $5$. Then $\Gamma$ is closed distance magic if and only if $\Gamma$ is isomorphic to $\Cay(\ZZ_n;\{\pm 1, \pm c, n/2\})$ with $n$ even and $1 < c < n/2$, and one of the following (i)--(iv) holds:
	\begin{itemize}
		\item[(i)] 
		$c=n/2-1$;
		\item[(ii)] 
		$n \equiv 2 \pmod{4}$, $c$ is even, and $2(c^2-1)$ is an odd multiple of $n$;
		\item[(iii)] 
		$n=3 \cdot 2^t(6k+(-1)^t)$ and $c=2^{t-1}(6k+(-1)^t)-1$ for some integer $t \ge 2$ and some integer $k \ge 0$ such that $c \ge 2$;
		\item[(iv)] 
		$n=3 \cdot 2^t(6k-(-1)^t)$ and $c=2^{t-1}(6k-(-1)^t)+1$ for some integer $t \ge 2$ and some integer $k \ge 0$ such that $c \ge 2$.
	\end{itemize}
\end{theorem}

 \section{Preliminaries}\label{sec:prelim}
 \noindent
 
 In this section we review basic definitions and results regarding closed distance magic graphs. 
 
 \subsection{Closed distance magic graphs and Cayley graphs}

Let $\Gamma=(V,E)$ denote a graph with vertex set $V$ and edge set $E$. Let $n=|V|$ denote the order of $\Gamma$. If vertices $x,y \in V$ are adjacent, then we denote this by $x\sim y$. For any $x\in V$, the \emph{open neighborhood} of $x$ is the set $N_{\Gamma}(x) := \left\{ y \in V :\ x\sim y \right\}$.
The \emph{closed neighborhood} of $x$ is the set $N_{\Gamma}[x] := N_{\Gamma}(x) \cup \left\{x\right\}$. We abbreviate $N_{\Gamma}(x)=N(x)$ and $N_{\Gamma}[x]=N[x]$ when $\Gamma$ is clear from the context. A bijective  labeling $\ell: V \to \{1,2, \ldots, n\}$ is called \emph{closed distance magic}, if  the number
\begin{align}
	\label{eq0}
	r = \sum_{y\in N[x]} \ell(y)
\end{align}
is independent of  the vertex $x\in V$. If $\Gamma$ admits a closed distance magic labeling, then we say that $\Gamma$ is \emph{closed distance magic (CDM)}. In this case we refer to the number $r$ from \eqref{eq0} as \emph{closed magic constant} of $\Gamma$.

In this paper, we will study close distance magic Cayley graphs of cyclic groups. Let $G$ be a finite group and let $S$ be an inverse-closed subset of $G$, which does not contain the identity element of $G$. Recall that the {\em Cayley graph} $\Cay(G;S)$ of $G$ with respect to the {\em connection set} $S$ is the graph with vertex-set $G$ in which $g, h \in G$ are adjacent if and only if $h = gs$ for some $s \in S$. Moreover, the graph $\Cay(G;S)$ is regular with valency $|S|$ and is connected if and only if $\la S \ra = G$. In the case that the group $G$ is cyclic the graph $\Cay(G;S)$ is called a {\em circulant}.  For any $n\geq 2$, the cyclic group $\ZZ_n$ is the group consisting of all congruence classes of $\ZZ$ modulo $n$. By abuse of notation, we will consider the elements of $\ZZ_n$ to be the numbers $\{0,1,\ldots,n-1\}$ and we will take the remainder modulo $n$, whenever it is needed in our computations.

\subsection{Regular closed distance magic graphs}

In this subsection we assume the graph $\Gamma$ is regular with valency $\kappa$. We recall certain results that  link the property of being closed distance magic to eigenvectors for the (potential) eigenvalue $-1$ of the corresponding adjacency matrix. Let $A$ denote the adjacency matrix of $\Ga$. The {\em eigenvalues of} $\Ga$ are the eigenvalues of its adjacency matrix $A$. Similarly, the \emph{spectrum} of $\Ga$ is the spectrum of its adjacency matrix $A$, that is, the multiset consisting of all of its eigenvalues.

\begin{proposition}
	\label{constant r}
	(\cite[Observation 2.1 and Corollary 2.3]{anholcer2016spectra}) Let $\Gamma$ denote a $\kappa$-regular graph on $n$ vertices. Assume $\Gamma$ is closed distance magic with closed magic constant $r$. Then $r=\frac{(\kappa+1)(n+1)}{2}$ and $-1$ is an eigenvalue of $\Gamma$.
\end{proposition}

The next theorem gives us a characterization of closed distance magic regular graphs in terms of their eigenvalues and eigenvectors.

\begin{theorem}\label{-1evalue}
	Let $\Gamma$ denote a $k$-regular graph on $n$ vertices and let $A$ denote its adjacency matrix. Then $\Gamma$ is closed distance magic if and only if there exists an eigenvector $\bfv$  of $A$ with eigenvalue $-1$, such that a certain permutation of its entries results in the following arithmetic sequence:
	\begin{equation}
		\label{eq1}
		\frac{1-n}{2}, \cdots,  \frac{2i-1-n}{2}, \cdots, \frac{n-1}{2}.
	\end{equation} 
	In particular, if $\Gamma$ is closed distance magic, then there exists an eigenvector corresponding to eigenvalue $-1$ for which all entries are pairwise distinct.  
\end{theorem}

\begin{proof}
	Let $V$ denote the vertex set of $\Gamma$ and let $I$ denote the identity matrix of order $n$. 
	
	Suppose $\Gamma$ is closed distance magic with closed magic constant $r$ and with closed distance magic labeling $\ell: V \to \{1, 2,\cdots,n\}$. Recall that $r=(\kappa+1)(n+1)/2$. Let $\bfv$ be the column vector of $\CC^n$ whose rows are indexed by the elements of $V$ such that the $x$-entry of $\bfv$ is equal to
	\begin{equation}\label{eq2}
		\ell(x)-\frac{r}{\kappa+1}=\ell(x)-\frac{n+1}{2}.
	\end{equation}   
	Since $\Gamma$ is regular with valency $\kappa$ and $\ell$ is a closed distance magic labeling with closed magic constant $r$, it is easy to see the $x$-entry of $(A+I)\bfv=0$ for every $x\in V$. Therefore, it follows that $\bfv$ is an eigenvector of the adjacency matrix of $\Gamma$ with $-1$ as the corresponding eigenvalue. That a certain permutation of the entries of $\bfv$ results in \eqref{eq1} is clear from \eqref{eq2} above. This also shows that the entries of $\bfv$ are pairwise distinct. 
	
	Conversely, assume  that $-1$ is an eigenvalue of $A$ and there exists an eigenvector $\bfv$ for the eigenvalue $-1$ with the property that a certain permutation of its entries results in the arithmetic sequence given in \eqref{eq1}. Let $\ell: V \to \{1,2,\cdots,n\}$ be the mapping defined by $$\ell(x)=\bfv_x+\frac{n+1}{2}$$ where $\bfv_x$ denotes the $x$-entry of $\bfv$ for $x\in V$.  Note that the assumption on $\bfv$ implies that $\ell$ maps $V$ to $\{1,2, \cdots,n\}$ and this map is a bijection. Moreover, since for every $x\in V$ the $x$-entry of the column vector $(A+I)\bfv$ equals $0$, it follows that 
	$$
	   \displaystyle \sum_{y\in N\left[ x\right]}\ell(y) =
	   \sum_{y\in N\left[ x\right]} \Big( \bfv_y + \frac{n+1}{2} \Big) = \frac{(\kappa+1)(n+1)}{2}.
	   \displaystyle
	 $$
	Therefore, $\Ga$ is closed distance magic, and this completes the proof. 
\end{proof}

\subsection{Representation theory of $\ZZ_n$}
In this subsection, we will recall the irreducible representations of the cyclic group $\ZZ_n$.  Given a group $G$, recall that a representation of $G$ is a group homomorphism $\mathfrak{X}$ from group $G$ to $\GL_n(\CC)$, the general linear group of degree $n$ and over $\CC$, for some $n\geq 1$. This positive integer $n$ is called the \emph{dimension} of the representation. It is not hard to see that there is a correspondence between representations of $G$ and submodules of the group algebra $\CC G$ (see \cite{dummit2004abstract}). We say that a representation of $G$ is \emph{irreducible} if the corresponding submodule of $\CC G$ is irreducible, that is, its only submodules are itself or the trivial one. 

Given a representation $\mathfrak{X}$ of $G$, the corresponding \emph{character} is the map $\chi: G\to \CC$ such that $\chi(g) = \operatorname{Trace}(\mathfrak{X}(g))$, for any $g\in G$. If a representation  of $G$ is irreducible, then we say that the corresponding character is \emph{irreducible}. We will denote the set of all non-equivalent irreducible characters of $G$ by $\irr{G}$.

Now let us consider the representations of cyclic groups. It is a well-known fact in the representation theory of finite groups that irreducible representations of an abelian group are one-dimensional. Hence, representations and characters coincide for these types of groups. As cyclic groups are abelian, the irreducible characters of $\ZZ_n$ are all possible homomorphisms from $\ZZ_n$ to $\CC^*$, the multiplicative group of non-zero complex numbers. It is not hard to see that in fact an irreducible character of $\ZZ_n$ is a homomorphism $\ZZ_n \to \left\{ z\in \CC :\ z^n = 1 \right\}$. Let $\mathbf{i}$ be the complex number such that $\mathbf{i}^2 =-1$. For any $j\in \{ 0,1,\ldots,n-1 \}$, define the map
$$
	\chi_{j}: \ZZ_n  \to \left\{ z\in \CC :\ z^n = 1 \right\}, \qquad
	\chi_{j}(x) = \cos\left(\frac{2\pi xj}{n}\right) + \mathbf{i} \sin\left(\frac{2\pi xj}{n}\right).%
$$
It is straightforward that for $0\leq j\leq n-1$ the map $\chi_{j}$ is a representation of $\ZZ_n$, and that
\begin{align}
	\irr{\ZZ_n} &= \left\{ \chi_j :\ 0\leq j\leq n-1 \right\}.\label{eq:irr}
\end{align}

\subsection{Eigenvalues of Cayley graphs over cyclic groups}

Computing the spectrum of a graph is a hard problem in general, even for Cayley graphs. However, the spectra of  Cayley graphs over abelian groups can be determined using representation theory of the underlying group.  
\begin{lemma}
	\label{lem:eigenvalues-cayley-graphs}
	(\cite{babai1979spectra}) Let $G$ be  an abelian group with identity element $0$ and let $S \subset G\setminus \{0\}$ such that $-x\in S$, whenever $x\in S$. The eigenvalues of the Cayley graph $\Cay(G;S)$ are of the form
	\begin{align*}
		\chi(S) := \sum_{x\in S} \chi(x),
	\end{align*}
	where $\chi$ runs through all the elements of $\irr{G}$. If $G = \{ g_1 = 0,g_2,\ldots,g_n \}$, then the vector 
	\begin{align*}
		\bfv_\chi &= \left[\chi(0),\chi(g_2),\ldots,\chi(g_n)\right]^t
	\end{align*} 
	is an eigenvector of $\Cay(G;S)$ corresponding to the eigenvalue $\chi(S)$.
	The dimension of the eigenspace corresponding to an eigenvalue $\lambda$ of $\Cay(G;S)$ is 
		$\left| \left\{ \chi \in \irr{G}:\ \chi(S) = \lambda  \right\} \right|$.
\end{lemma}

Consequently, the eigenvalues of the circulant graph $\Cay(\ZZ_n;S)$ are of the form
\begin{align}
	\chi_j(S) := 
	 \sum_{x\in S} \left( \cos\left(\frac{2\pi xj}{n}\right) + \mathbf{i} \sin\left(\frac{2\pi xj}{n}\right)\right),
\end{align}
where $j$ runs through the elements of $\{0,1,\ldots,n-1\}$. For $0\leq j\leq n-1$, let 
\begin{align}
	\bfv_j :=\bfv_{\chi_j} = \left[1,\omega^j,\omega^{2j},\ldots,\omega^{j(n-1)}\right]^t,\label{eq:eigenvectors-general}
\end{align}
where $\omega =  \cos\left(\frac{2\pi }{n}\right) + \mathbf{i} \sin\left(\frac{2\pi }{n}\right)$.

Using Lemma~\ref{lem:eigenvalues-cayley-graphs}, we deduce that the set of vectors $\{\bfv_j :\ 0\leq j\leq n-1 \}$ is a complete set of basic eigenvectors. Using the above comments and the fact that $\chi_j(n/2)=(-1)^j$, we obtain the following result about the eigenvalues of circulant graphs.

\begin{lemma}
	\label{lem:eigenvalues}
	\begin{enumerate}[label=\rm(\roman*)]
		\item Let $n\geq 4$ be even and $1\leq a < \frac{n}{2}$. The eigenvalues of the cubic circulant graph $\Cay(\ZZ_n;\{\pm a, \frac{n}{2}\})$ are
		\begin{align*}
			2\cos\left( \frac{2ak\pi }{n} \right) + (-1)^k  \qquad (0 \le k \le n-1).
		\end{align*}
	
		\item Let $n\geq 5$ be an integer and $1\leq a < b < \frac{n}{2}$. The eigenvalues of the tetravalent circulant graph $\Cay(\ZZ_n;\{\pm a, \pm b \})$ are
		\begin{align*}
			2\cos\left( \frac{2ak\pi }{n} \right) +2\cos\left( \frac{2bk\pi }{n} \right)  \qquad (0 \le k \le n-1).
		\end{align*}
	
	\item Let $n\geq 6$ be even and $1\leq a < b<\frac{n}{2}$. The eigenvalues of the $5$-valent circulant graph $\Cay(\ZZ_n;\{\pm a, \pm b, \frac{n}{2} \})$ are
	\begin{align*}
		2\cos\left( \frac{2ak\pi }{n} \right) + 2\cos\left( \frac{2bk\pi }{n} \right) + (-1)^k  \qquad (0 \le k \le n-1).
	\end{align*}
		
		\end{enumerate} 
	
\end{lemma}

\subsection{Admissible characters for CDM circulants}
By Theorem~\ref{-1evalue}, we know that a regular closed distance magic graph $\Ga$ admits an eigenvalue $-1$ with a certain eigenvector. In case when $\Ga$ is a circulant graph, we have the following definition. 

\begin{definition}
	\label{def:adm}
	Let $n\geq 3$ and let $S$ be an inverse-closed subset of $\ZZ_n \setminus \{0\}$. Consider the graph $\Cay(\ZZ_n;S)$. If for some integer $0 \le j \le n-1$ we have that $\chi_j(S) = -1$, then we call integer $j$ (as well as the corresponding irreducible character $\chi_j$) admissible. Define
	\begin{align*}
		\mathcal{J}_n(S) := \{ j\in \{ 0,1,\ldots,n-1\} :\ \mbox{$j$ is admissible} \}.
	\end{align*}
\end{definition}

With reference to Definition \ref{def:adm}, recall that vector $\bfv_j$  from \eqref{eq:eigenvectors-general} is an eigenvector of $\Cay(\ZZ_n;S)$ corresponding to the eigenvalue $\chi_j(S)$. Recall also that the subspace
$\operatorname{Span}\{ \bfv_j :\ j\in \mathcal{J}_n(S) \}$
is the eigenspace corresponding to the eigenvalue $-1$ of  $\Cay(\ZZ_n;S)$. The following result will be crucial for the rest of this paper. 

\begin{proposition}
	Let $n\geq 3$ and let $S$ be an inverse-closed subset of $\ZZ_n \setminus \{0\}$. Consider the graph $\Cay(\ZZ_n;S)$. If there exist distinct $x,y \in \ZZ_n$ such that $\chi_j(x) = \chi_j(y)$ for all $j\in \mathcal{J}_n(S)$, then $\Cay(\ZZ_n;S)$ is not closed distance magic.\label{prop:eigenspace}
\end{proposition}
\begin{proof}
	Assume that $x,y$ are distinct elements of $\ZZ_n$, such that $\chi_j(x) = \chi_j(y)$ for all $j\in \mathcal{J}_n(S)$. Then all eigenvectors $\bfv_j \; (j \in \mathcal{J}_n(S))$ have the $x$-entry equal to the $y$-entry. Consequently, every eigenvector for the eigenvalue $-1$ have the $x$-entry equal to the $y$-entry. By Theorem~\ref{-1evalue}, $\Cay(\ZZ_n;S)$ is not closed distance magic.
\end{proof}

\subsection{Trigonometric equation}
\label{subsec:trig}
In 1944, H.~S.~M.~Coxeter posed the following problem: determine all rational solutions of the equation
\begin{equation}
	\label{eq:Cox}
	\cos(r_1\pi) + \cos(r_2\pi) + \cos(r_3\pi) = 0,\quad 0 \leq r_1 \leq r_2 \leq r_3 \leq 1.
\end{equation}
The problem was solved in 1946 by W.~J.~R.~Crosby~\cite{CoxCro46}. It was proved that, except for a pair of ``symmetric'' exceptions, the only solutions of~\eqref{eq:Cox} are those that belong to two infinite families of ``obvious'' triples $(r_1, r_2, r_3)$, namely 
\begin{equation}
	\label{eq:sol1}
	0 \leq r_1 \leq \frac{1}{2},\quad r_2 = \frac{1}{2},\quad r_3 = 1 - r_1,
\end{equation}
and
\begin{equation}
	\label{eq:sol2}
	0 \leq r_1 \leq \frac{1}{3},\quad r_2 = \frac{2}{3}-r_1,\quad r_3 = \frac{2}{3} + r_1.
\end{equation}
The only two exceptions are 
\begin{equation}
	\label{eq:sol3}
	r_1 = \frac{1}{5},\ r_2 = \frac{3}{5},\ r_3 = \frac{2}{3}\quad \text{and}\quad r_1 = \frac{1}{3},\ r_2 = \frac{2}{5},\ r_3 = \frac{4}{5}.
\end{equation}

It is clear that no triple $(r_1,r_2,r_3)$ of rational numbers with $0 \leq r_1 \leq r_2 \leq r_3 \leq 1$ which satisfies any of the two possibilities from~\eqref{eq:sol3} satisfies~\eqref{eq:sol1} or~\eqref{eq:sol2}. Moreover, the only triple $(r_1,r_2,r_3)$ which satisfies both~\eqref{eq:sol1} and~\eqref{eq:sol2} is $(\frac{1}{6}, \frac{1}{2}, \frac{5}{6})$.

\section{Proof of Theorem~\ref{thm:val3,4}} 

In this section we prove Theorem~\ref{thm:val3,4}. To do this, we will use Proposition \ref{prop:eigenspace} extensively. We start with the cubic case. We would like to point out that we could prove this result also using more elementary methods. However, to demonstrate our approach, we will prove it using machinery developed in Section \ref{sec:prelim}. 

\subsection{Proof of Theorem~\ref{thm:val3,4} - cubic case}
	
	Note that the complete graph $K_4$ is clearly a closed distance magic circulant graph. For the other direction, assume that $n$ is even and let $1 \le a <n/2$. Let $S = \{ \pm a,\frac{n}{2} \}$ and define  $\Gamma := \Cay(\ZZ_n;S_a)$. Suppose that $\Gamma$ is a connected closed distance magic graph. By Lemma~\ref{lem:eigenvalues}(i), we know that the eigenvalues of $\Ga$ are 
	\begin{align*}
		\chi_j(S) &= 2\cos\left(\frac{2\pi ja}{n}\right)+(-1)^j \qquad (0 \le j \le n-1).
	\end{align*}
	By Theorem~\ref{-1evalue}, we must have $\chi_j(S)=-1$ for some $0 \le j \le n-1$, and so $\mathcal{J}_n(S)$ is nonempty. Depending on the parity of $j \in \mathcal{J}_n(S)$ we have the following two possibilities.
If $j \in \mathcal{J}_n(S)$ is odd, then Lemma \ref{lem:eigenvalues}(i) implies that $\cos\left(\frac{2\pi ja}{n}\right)= 0$, and so $4ja=n(2t+1)$ for some $t \in \mathbb{Z}$.	
If  $j \in \mathcal{J}_n(S)$ is even, then Lemma \ref{lem:eigenvalues}(i) implies that $\cos\left(\frac{2\pi ja}{n}\right)=-1$, and so $2ja=n(2t+1)$ for some $t \in \mathbb{Z}$.	Therefore, for each $j \in \mathcal{J}_n(S)$  we have that $4ja$ is a multiple of $n$, implying that $\chi_j(4a)= 1$ for every $j \in \mathcal{J}_n(S)$. As $\chi_j(0)=1$ for every $0 \le j \le n-1$, Proposition~\ref{prop:eigenspace} yields that $4a=0$ holds in $\ZZ_n$. As $a < n/2$, this implies $a=n/4$, and so $S=\{n/4,n/2,3n/4\}$. Since $\Ga$ is connected, $S$ must generate $\ZZ_n$, forcing $n=4$. Consequently, $\Ga$ is isomorphic to the complete graph $K_4$.  This completes the proof. \hfill \Qed

\subsection{Proof of Theorem~\ref{thm:val3,4} - tetravalent case}

Again, it is clear that the complete graph $K_5$ is a closed distance magic circulant graph. For the other direction, let $1\leq a < b <\frac{n}{2}$ and $S=\{\pm a, \pm b\}$. Define $\Gamma := \Cay(\ZZ_n;S)$ and assume that $\Ga$ is a connected closed distance magic graph. Observe that by Proposition~\ref{constant r} we have $r=\frac{5(n+1)}{2}$, and so $n$ is odd. By Lemma \ref{lem:eigenvalues}(ii) we have that
	\begin{align*}
		\cos\left(\frac{2\pi ja}{n}\right)+\cos\left(\frac{2\pi jb}{n}\right)=-\frac{1}{2}
	\end{align*}
for every $j\in \mathcal{J}_n(S)$.
	Therefore, such admissible $j$ must satisfy
	\begin{align}\label{eq4}
		\cos\frac{\pi}{3} + \cos\left(\frac{2\pi ja}{n}\right)+\cos\left(\frac{2\pi jb}{n}\right)=0.
	\end{align}
	Possible solutions for equation \eqref{eq4} are described in Subsection \ref{subsec:trig}. We analyze the solutions of \eqref{eq4} as follows.
	     Suppose first that \eqref{eq4} admits a solution of type~\eqref{eq:sol1}. Then, we must have
		$$
		  \frac{2ja}{n}=\frac{1}{2} + t \,\,\, \mbox{or}\,\,\, \frac{2jb}{n}=\frac{1}{2} + t
		$$ 
		for some integer $t$. It follows that $n(2t+1) \in \{4ja, 4jb\}$, contradicting the fact that $n$ is odd.
		Next, suppose \eqref{eq4} admits a solution of type~\eqref{eq:sol2}. Then, either $2 \pi ja/n$ or $2 \pi jb/n$ is equal to $\pi + 2t\pi$ for some integer $t$. But this implies that $n(1+2t) \in \{2ja, 2jb\}$, again contradicting the fact that $n$ is odd. 
		It follows that for every $j\in \mathcal{J}_n(S)$, the corresponding solution of  \eqref{eq4} is of type~\eqref{eq:sol3}. Therefore,
		$$
		 \Big\{{2 \pi ja \over n}, {2\pi j b \over n} \Big\} =\Big\{ \pm {2 \pi \over 5} + 2 t_1 \pi, \pm {4 \pi \over 5} + 2 t_2 \pi \Big\}
		$$ 
		for some integers $t_1, t_2$. It follows that $\chi_j(5a)=\chi_j(5b)=1$ for any $j\in \mathcal{J}_n(S)$. As $\chi_j(0)=1$ for every $0 \le j \le n-1$, Proposition~\ref{prop:eigenspace} implies that $5a=5b=0$ holds in $\ZZ_n$. As $1 \le a < b \le n/2$ this forces $a=n/5$ and $b=2n/5$, and so $S=\{n/5,2n/5, 3n/5, 4n/5\}$. Since $\Ga$ is connected, $S$ must generate $\ZZ_n$, forcing $n=5$. Consequently, $\Ga$ is isomorphic to the complete graph $K_5$.  \hfill \Qed

	\section{Proof of Theorem \ref{thm:val5} - A necessary condition}
	In this section we prove that every connected closed distance magic circulant graph with valency $5$ is isomorphic to a graph belonging to one of the four families described in Theorem \ref{thm:val5}. To do this, we will use the following notation.
	
	\begin{notation}
		\label{gen}
		Let $n \ge 6$ be an even integer and let $1 \le  a < b < n/2$. Let $\Gamma:= \Cay(\ZZ_n;\{\pm a, \pm b, \frac{n}{2}\})$ and assume that $\Ga$ is closed distance magic. Write $n = 2^t \ell$, $b+a = 2^\alpha \ell_1$ and $b-a = 2^\beta \ell_2$, where $t, \alpha, \beta$ are non-negative integers and $\ell, \ell_1, \ell_2$ are odd positive integers. For $i \in \{1,2\}$ let $d_i = \gcd(\ell,\ell_i)$ and let $n_i, m_i$ be positive integers such that $\ell = d_i n_i$ and $\ell_i = d_i m_i$. Let $\mathcal{J}_n(S)$ be as in Definition \ref{def:adm} and note that $\mathcal{J}_n(S)$ is nonempty. Observe also that by Proposition~\ref{constant r} we have $r=3(n+1)$. For an integer $m$ and a prime $p$, we let the {\em $p$-part} of $m$ be $p^t$, where $t$ is the largest integer such that $m$ is divisible by $p^t$.
    \end{notation}

 Pick $j \in \mathcal{J}_n(S)$ and note that by Lemma \ref{lem:eigenvalues}(iii) we have 
 \begin{align*}
 	2\cos\left(\frac{2\pi ja}{n}\right)+2\cos\left(\frac{2\pi jb}{n}\right) + (-1)^j=-1.
 \end{align*}   
If $j$ is even, then the above equality is equivalent to 
\begin{align}\label{eq5}
	\cos\left(\frac{2\pi ja}{n}\right)+\cos\left(\frac{2\pi jb}{n}\right) + \cos 0 = 0,
\end{align}
while if $j$ is odd, then the above equality is equivalent to 
\begin{align}\label{eq6}
	\cos\left(\frac{2\pi ja}{n}\right)+\cos\left(\frac{2\pi jb}{n}\right) =0.
\end{align}
Assume for a moment that $j \in \mathcal{J}_n(S)$ is even. Then it is clear that a solution of \eqref{eq5} could not be as described in \eqref{eq:sol3}. Moreover, if a solution of \eqref{eq5} is as described in \eqref{eq:sol1}, then we have that $\{\frac{2\pi ja}{n},\frac{2 \pi jb}{n}\} = \left\{ \pi+2k_1\pi, \frac{\pi}{2}+k_2\pi\right\}$  for some $k_1,k_2 \in \mathbb{Z}$. In this case we say that $j$ (as well as the corresponding character $\chi_j$) is of {\em type} 1. If, however,  a solution of \eqref{eq5} is as described in \eqref{eq:sol2}, then we have that $\{\frac{2\pi ja}{n},\frac{2\pi jb}{n}\} =  \left\{ \pm\frac{2\pi}{3}+2k_1\pi, \pm\frac{2\pi}{3}+2k_2\pi\right\}$ for some $k_1,k_2 \in \mathbb{Z}.$ In this case we say that $j$ (as well as the corresponding character $\chi_j$) is of {\em type} 2. 

Assume now that $j \in \mathcal{J}_n(S)$ is odd. Then \eqref{eq6} implies that $\{\frac{2\pi j b}{n} \} = \{  \pi \pm \frac{2\pi j a}{n}+2k\pi\}$ for some $k \in \mathbb{Z}$. Consequently, either $j=n(2k+1)/(2(b+a))$, or $j=n(2k+1)/(2(b-a))$. In the former case we say that $j$ (as well as the corresponding character $\chi_j$) is of {\em type} $3^+$, while in the latter case we say that $j$ (as well as the corresponding character $\chi_j$) is of {\em type} $3^-$. We continue our analysis with the following lemma in which we gather some properties of numbers $a,b$, and admissible $j$'s.  

\begin{lemma}
	\label{lem:one}
With reference to Notation \ref{gen}, the following {\rm{(i)-(viii)}} hold.

\begin{enumerate}[label=\rm(\roman*)]
	\item There exists at least one $j \in \mathcal{J}_n(S)$  of type $3^+$ or $3^-$. 
	\item If $j \in \mathcal{J}_n(S)$  is of type $3^+$, then $j=n_1 (2s_1+1)$ for some integer $s_1$. Similarly, if $j \in \mathcal{J}_n(S)$  is of type $3^-$, then $j=n_2 (2s_2+1)$ for some integer $s_2$.
	\item If there exists $j \in \mathcal{J}_n(S)$  of type $3^+$ (type $3^-$, respectively), then $t =\alpha+1$ ($t=\beta+1$, respectively). 
	\item At least one of $a,b$ is odd.
	\item If one of $a,b$ is odd and the other one is even, then $t=1$.
	\item If there exists $j \in \mathcal{J}_n(S)$ which is of type $2$, then $n$ is divisible by $3$.
	\item There are no $j \in \mathcal{J}_n(S)$ of type 1.
	\item At least one of $a,b$ is relatively prime to $n$.
	\end{enumerate}

\end{lemma}
\begin{proof}
	(i) If all $j \in \mathcal{J}_n(S)$  are even, then $\chi_j(n/2)=1$ for every $j \in \mathcal{J}_n(S)$. Taking $x=0$ and $y=n/2$ in Proposition \ref{prop:eigenspace}, we get that $\Ga$ is not closed distance magic, a contradiction. 
	
	\noindent
	(ii) We prove the claim for $j \in \mathcal{J}_n(S)$  of type $3^+$ (the proof of type $3^-$ is similar). Recall that by the comments following Notation \ref{gen} we have 
	$$
	  j = {n(2k+1) \over 2(b+a)} = {2^{t-1} d_1 n_1 (2k+1) \over 2^{\alpha} d_1 m_1}
	$$
	for some integer $k$. As $j$ is odd, we have that $\alpha=t-1$, and so  $j=n_1 (2k+1)/m_1$. As $\gcd(n_1,m_1)=1$, $m_1$ must divide $2k+1$, and the result follows. 
	
	\noindent
	(iii) The claim follows from (i) above and from the proof of (ii) above. 
	
	\noindent
	(iv) If $a,b$ are both even, then also $b+a$ and $b-a$ are even, and so $\alpha, \beta$ are both greater or equal to $1$. By (iii) above we have $t \ge 2$, and so $n/2$ is also even, contradicting the fact that $\Ga$ is connected.
	
	\noindent
	(v) If one of $a,b$ is odd and the other one is even, then $b+a$ and $b-a$ are both odd, and so $\alpha=\beta=0$. It follows from (ii) above that $t=1$.
	
	\noindent
	(vi) Let $j \in \mathcal{J}_n(S)$ be of type 2. Recall that in this case we have that $3ja=n(3k_1 \pm 1)$ and $3jb=n(3k_2 \pm 1)$ hold for some integers $k_1, k_2$. This shows that $n$ is divisible by $3$. 
	
	\noindent
	(vii) Assume that $j \in \mathcal{J}_n(S)$ is of type 1. By the comments following Notation \ref{gen}, we know that in this case
	\begin{equation}
		\label{eq:type1}
		\{2ja, 2 jb\} = \left\{ n(1+2k_1), \frac{n(1+2k_2)}{2}\right\}
	\end{equation}
	for some $k_1,k_2 \in \mathbb{Z}$. Note that as $j$ is even, this implies that $n$ is divisible by $8$, that is, $t \ge 3$. It follows from (iv), (v) above that $a,b$ are both odd, and so $2ja$ and $2jb$ have the same $2$-part. But this contradicts \eqref{eq:type1}, as $n(1+2k_1)$ and $n(1+2k_2)/2$ clearly do not have the same $2$-part.
	
	\noindent
	(viii) Recall that by (iv) above at least one of $a,b$ is odd.  Assume that $a$ is odd. We show that $\gcd(a,n)=1$ (the case when $b$ is odd is treated similarly). Denote $d=\gcd(n,a)$ and assume to the contrary that $d \ge 2$. Observe that $d$ is odd, and let $p$ be an odd prime dividing $d$. Note that $p$ also divide $n/2$, and so connectedness of $\Ga$ implies that $b$ is not divisible by $p$. In particular, none of $b+a$, $b-a$ is divisible by $p$. 
	
	Pick $j \in \mathcal{J}_n(S)$  which is of type $3^+$ or $3^-$. Recall that in this case we have $j(b \pm a) =n(2k+1)/2$ for some integer $k$. As $p$ divides $n/2$ but does not divide $b \pm a$, it must divide $j$.
	
	Let now $j \in \mathcal{J}_n(S)$ be of type 2. As in the case (vi) above we have that $3ja=n(3k_1 \pm 1)$, $3jb=n(3k_2 \pm 1)$ holds for some integers $k_1, k_2$, and $n$ is divisible by $3$. Write $n=3n_0$, and so $ja=n_0(3k_1 \pm 1)$ and $jb=n_0(3k_2 \pm 1)$. Observe that we either have $j(b+a)=3 n_0 (k_2+k_1) = n(k_2+k_1)$, or  $j(b-a)=3 n_0 (k_2-k_1) = n(k_2-k_1)$. As $p$ divides $n$ but does not divide $b \pm a$, it must divide $j$.
	
	Therefore, we just showed that every admissible $j \in \mathcal{J}_n(S)$ is divisible by $p$, which implies that for every $j \in \mathcal{J}_n(S)$ we have $\chi_j(n/p)=1$, contradicting Proposition~\ref{prop:eigenspace}. 
\end{proof}

It is well known and easy to see that for any $q \in \ZZ_n$ with $\gcd(q,n) = 1$ the graph $\Ga$ is isomorphic to $\Cay(\ZZ_n ; \{\pm qa, \pm qb, n/2\})$ (as $q$ is odd, we have that $qn/2=n/2$ holds in $\ZZ_n$). By Lemma \ref{lem:one}(viii), at least one of $a,b$ is relatively prime to $n$. Taking $q$ to be the multiplicative inverse of this element which is coprime to $n$, we get that $\Ga$ is isomorphic to   $\Cay(\ZZ_n; \{ \pm 1, \pm c, n/2\})$ with $c < n/2$.  Therefore, we will be using Notation \ref{gen} with the additional convention that $a=1$ and $b=c$, for the rest of this section.

\begin{proposition}
	\label{prop:type3-only}
	With reference to Notation \ref{gen}, there exists $j \in \mathcal{J}_n(S)$  which is of type $2$ or of type $3^+$.
\end{proposition}
\begin{proof}
	Assume to the contrary that all $j \in \mathcal{J}_n(S)$ are of type $3^-$. Pick $j \in \mathcal{J}_n(S)$  and recall that we have $2j(c-1)=n(2k+1)$ for some integer $k$. Therefore,
	$$
	  \chi_j(2(c-1)) = \cos{\Big( \frac{4 \pi  j (c-1)}{n} \Big)} + \mb{i} \sin{\Big( \frac{4 \pi  j (c-1)}{n} \Big) }= \cos(2 \pi(2k+1)) + \mb{i} \sin(2 \pi(2k+1)) =1.
	 $$
	 By Proposition \ref{prop:eigenspace} we therefore have that $2(c-a)=0$ holds in $\ZZ_n$. This means that integer $2(c-a)$ is a multiple of $n$. But as $2 \le c < n/2$, we have that $2 \le 2(c-1) \le 2(n/2-2)$, a contradiction. 
\end{proof}

\begin{proposition}
	\label{prop:type3+only}
	With reference to Notation \ref{gen}, assume that all $j \in \mathcal{J}_n(S)$  are of type $3^+$. Then, $\Ga$ is isomorphic to a graph belonging to the family described in part (i) of Theorem \ref{thm:val5}. 
\end{proposition}
\begin{proof}
	Assume that all $j \in \mathcal{J}_n(S)$ are of type $3^+$. Pick $j \in \mathcal{J}_n(S)$  and recall that we have $2j(c+1)=n(2k+1)$ for some integer $k$. Similarly, as in the proof of Proposition \ref{prop:type3-only} we get that $\chi_j(2(c+1)) = 1$, and so by Proposition \ref{prop:eigenspace} we have that $2(c+1)=0$ holds in $\ZZ_n$. This means the integer $2(c+1)$ is a multiple of $n$. But $1 < c < n/2$ implies that $2(c+1)=n$, and so $c=n/2-1$. 
	\end{proof}

With reference to Notation \ref{gen}, for the rest of this section we assume that there exists $j \in \mathcal{J}_n(S)$ which is of type $2$. Recall that in this case $n$ is divisible by $3$ and there exist at least one $j \in \mathcal{J}_n(S)$  which is of type $3^+$ or $3^-$. 

\begin{lemma}
	\label{lem:type2}
	With reference to Notation~\ref{gen}, assume that there exists $j \in \mathcal{J}_n(S)$ which is of type $2$. Then, the only elements of $\mathcal{J}_n(S)$ which are of type $2$ are $n/3$ and $2n/3$. In particular, $c$ is not divisible by $3$.
\end{lemma}
\begin{proof}
	Let $j \in \mathcal{J}_n(S)$ be of type $3$. By the comments following Notation \ref{gen} we have that $3j=n(3k_1 \pm 1)$ and $3jc=n(3k_2 \pm 1)$ for some $k_1, k_2 \in \ZZ$. It follows that $j=n_0(3k_1 \pm 1)$, where $n_0=n/3$. As $0 \le j \le n-1$, this implies that $j \in \{n_0, 2n_0\}$. However, it is easy to see from \eqref{eq5} that $n_0 \in \mathcal{J}_n(S)$ if and only if $2n_0 \in \mathcal{J}_n(S)$. It is also clear from \eqref{eq5} that if $c$ is divisible by $3$, then $n_0 \not \in \mathcal{J}_n(S)$, a contradiction.
\end{proof}

\begin{lemma}
	\label{cond:lem1a}
	With reference to Notation~\ref{gen}, the integers $d_1$ and $d_2$ are coprime. 
\end{lemma}

\begin{proof}
	Note that by definition $d_1$ and $d_2$ are odd. Now, if $d$ divides $d_1$ (and thus $n$) and $d_2$, then it divides both $c+1$ and $c-1$, and so it also divides $2$ and $2c$. Since $d$ is odd, it follows that $d = 1$.
\end{proof}

\begin{proposition}
	\label{cond:prop3}
	With reference to Notation~\ref{gen}, assume that there exists $j \in \mathcal{J}_n(S)$ of type $2$. Then, $\ell = d_1d_2$. Consequently, $n = 2^t d_1 d_2$, $n_1=d_2$ and $n_2=d_1$. 
\end{proposition}

\begin{proof}
	We first claim that either $d_1$ or $d_2$ is divisible by $3$. Recall that by Lemma \ref{lem:one}(vi) $n$ is divisible by $3$, and that by Lemma \ref{lem:type2}, $c$ is not divisible by $3$.  Consequently, exactly one of $c+1, c-1$ is divisible by $3$. Assume that $c+1=2^{\alpha} d_1 m_1$ is divisible by $3$ (proof  of the claim in the case when $c-1$ is divisible by $3$ is similar). If $d_1$ is divisible by $3$, then we are done. If $m_1$ is divisible by $3$, then $n_1$ is not divisible by $3$ (recall that $n_1$ and $m_1$ are relatively prime). As $n=2^t d_1 n_1$ is divisible by $3$, we again have that $d_1$ is divisible by $3$. The claim follows. 
	
	Pick any $j \in \mathcal{J}_n(S)$. If $j$ is of type 2 then, by Lemma \ref{lem:type2}, we have that $j \in \{n/3, 2n/3\}$. As either $d_1$ or $d_2$ is divisible by $3$, we have that 
	$$
			\chi_{n/3}(2^t d_1 d_2)) = \cos{\Big( 2 \pi 2^t  (d_1 d_2/3) \Big)} + \mb{i} \sin{\Big( 2 \pi  2^t  (d_1 d_2/3) \Big)} = 1. 
	$$
	Similarly, we get that $\chi_{2n/3}(2^t d_1 d_2)=1$. If $j$ is of type $3^+$, then, using Lemma \ref{lem:one}(ii), we have  
	\begin{equation}
		\begin{split}
			\chi_j(2^t d_1 d_2)) & = \cos{\Big( \frac{2 \pi  j  2^t d_1 d_2}{n} \Big)} + \mb{i} \sin{\Big( \frac{2 \pi  j 2^t d_1 d_2}{n} \Big)} \\
			&= \cos{\Big( \frac{2 \pi  n_1 (2 s_1+1) 2^t  d_1 d_2}{2^t d_1 n_1} \Big)} + \mb{i} \sin{\Big( \frac{2 \pi  n_1 (2 s_1+1) 2^t  d_1 d_2}{2^t d_1 n_1} \Big)} \\
			&= \cos{\Big(2\pi (2 s_1+1)  d_2 \Big)} + \mb{i} \sin{\Big(2 \pi (2 s_1+1)  d_2 \Big)} = 1. 
		\end{split}
	\end{equation}
	A similar argument shows that also in the case when $j$ is of type $3^-$ we have that $\chi_j(2^t d_1 d_2)=1$.
	
	Therefore, for every $j \in \mathcal{J}_n(S)$ we have that $\chi_j(2^t d_1 d_2)=1$, and so by Proposition~\ref{prop:eigenspace} we have that $2^t d_1 d_2$ is a (positive) multiple of $n$. However, by Lemma~\ref{cond:lem1a} the integers $d_1$ and $d_2$ are relatively prime divisors of $\ell$, and so $d_1 d_2 \le \ell$. Note that this implies $2^t d_1 d_2 \le n$, and so we in fact have $2^t d_1 d_2 = n$ as claimed.
\end{proof}

We now analyze the case when $c$ is even.
 
 \begin{proposition}
 	\label{prop:c-even}
 	With reference to Notation~\ref{gen}, assume that there exists $j \in \mathcal{J}_n(S)$ of type $2$. If $c$ is even,  then $\Ga$ is isomorphic to a graph belonging to the family described in part (ii) of Theorem \ref{thm:val5}. 
 \end{proposition}
 \begin{proof}
 	Assume $c$ is even. Then $c+1$ and $c-1$ are odd, and so $\alpha=\beta=0$. Consequently, $t=1$ by Lemma \ref{lem:one}(iii), which means that $n \equiv 2 \pmod{4}$. Also $2(c^2-1)=2(c+1)(c-1)=2d_1 m_1d_2 m_2 = n m_1 m_2$, and so $2(c^2-1)$ is indeed an odd multiple of $n$.
 \end{proof}

Let us now turn our attention to the case when $c$ is odd. 
 \begin{lemma}
 	\label{lem:c-odd}
 	With reference to Notation~\ref{gen}, assume that there exists $j \in \mathcal{J}_n(S)$ of type $2$. If $c$ is odd,  then either all $j \in \mathcal{J}_n(S)$ are of types $2$ and $3^+$, or of types $2$ and $3^-$. 
 \end{lemma}
\begin{proof}
	Assume that $c$ is odd, and so $c+1$ and $c-1$ are two consecutive even integers. It follows that one of $c+1,c-1$ is divisible by $4$, and the other one is not. In other words, one of $\alpha, \beta$ is at least $2$, while the other one is equal to $1$. It follows now from Lemma \ref{lem:one}(iii) that there cannot exist $j_1, j_2 \in \mathcal{J}_n(S)$, such that $j_1$ is of type $3^+$, while $j_2$ is of type $3^-$. 
\end{proof}

\begin{proposition}
	\label{prop:type23+}
	With reference to Notation~\ref{gen}, assume that there exists $j \in \mathcal{J}_n(S)$ of type $2$ and $c$ is odd. Also let  all $j \in \mathcal{J}_n(S)$ be of type $2$ or $3^+$. Then, $\Ga$ is isomorphic to a graph belonging to the family described either in part (i) or in part (iii) of Theorem \ref{thm:val5}. 
\end{proposition}
\begin{proof}
Observe first that if $c=n/2-1$, then $\Ga$ is isomorphic to a graph belonging to the family described in part (i) of Theorem \ref{thm:val5}. For the rest of this proof we therefore assume that $c < n/2-1$.

 Recall that we have $n=2^td_1d_2$ and $t=\alpha+1$. Therefore, $c+1=2^{t-1} d_1 m_1$. Recall also that by Lemma \ref{lem:one}(ii), an admissible $j$ of  type $3^+$ is of the form $j=d_2(2s+1)$ for some integer $s$, and that an admissible $j$ of type $2$ either $n/3$ or $2n/3$. As $n$ is divisible by $3$ and $d_1, d_2$ are relatively prime, we have that $3$ divides exactly one of $d_1, d_2$. Assume first that $3 | d_1$. Then
 \begin{align}
 	\chi_{\frac{n}{3}}(2^t d_1)&=\cos\left(\frac{2\pi \frac{n}{3}2^td_1}{n}\right)+\mb{i} \sin \left(\frac{2\pi \frac{n}{3}2^td_1}{n}\right)=1.
 \end{align}
Similarly, $\chi_{\frac{2n}{3}}(2^t d_1)=1$. Now if $j \in \mathcal{J}_n(S)$ is of type $3^+$, then $j=d_2(2s+1)$ for some integer $s$. Therefore, as $n=2^t d_1 d_2$, we have
$$ 
  \chi_j(2^t d_1) = \cos\left(\frac{2\pi d_2 (2s+1) 2^td_1}{n}\right)+\mb{i} \sin \left(\frac{2\pi d_2 (2s+1) 2^td_1}{n}\right)=1.
$$
It follows that for each admissible character $\chi_j$ we have that $\chi_j(2^t d_1)=1$, and so Proposition~\ref{prop:eigenspace} implies that $d_2=1$.  However, since $d_1$ is a divisor of $\ell_1$,
$$
  c+1=2^{t-1}\ell_1< \frac{n}{2}=2^{t-1}d_1
$$
yields a contradiction.  

Now, suppose that $3\mid d_2$. Similarly, as above we see that for any admissible character $\chi_j$ we have that 
$\chi_{j}(2^t 3 d_1)=1$. Proposition~\ref{prop:eigenspace} implies that $2^t3d_1=kn=k2^td_1d_2$ for some $k\in\mathbb{Z}$. Therefore, $3=kd_2$, and since $3\mid d_2$ we must have $d_2=3$ and $k=1$. As $c+1=2^{t-1}\ell_1 < \frac{n}{2}=2^{t-1}3d_1$, we have that $\ell_1< 3d_1$. As $\ell_1$ is an odd multiple of $d_1$, this shows that $ \ell_1=d_1$. Summarizing all together, we have $n=2^t 3 d_1$, $c+1=2^{t-1} d_1$, and so $c-1=2^{t-1} d_1-2=2(2^{t-2} d_1-1)$ for some odd $d_1$, which is not divisible by $3$. But as $3|d_2$, we have that $3|c-1$, and so $3$ is a divisor of $2^{t-2} d_1 - 1$. If $t$ is even, then we have that $2^{t-2}\equiv1 \pmod{3}$, and so $d_1\equiv 1 \pmod{3}$. As $d_1$ is odd, $d_1=6k+1$ for some nonnegative integer $k$. If $t$ is odd, then we have that $2^{t-2}\equiv-1 \pmod{3}$, and so $d_1\equiv -1 \pmod{3}$. As $d_1$ is odd, $d_1=6k-1$ for some positive integer $k$. This shows that $\Ga$ is isomorphic to a graph belonging to the family described in part (iii) of Theorem \ref{thm:val5}.
\end{proof}

\begin{proposition}
	\label{prop:type23-}
	With reference to Notation~\ref{gen}, assume that there exists $j \in \mathcal{J}_n(S)$ of type $2$, and $c$ is odd. Also let all $j \in \mathcal{J}_n(S)$ be of type $2$ or $3^-$. Then, $\Ga$ is isomorphic to a graph belonging to the family described in part (iv) of Theorem \ref{thm:val5}. 
\end{proposition}
\begin{proof}
Recall that we have $n=2^td_1d_2$ and that $t=\beta+1$. Therefore $c-1=2^{t-1} d_2 m_2$. Recall also that by Lemma \ref{lem:one}(ii), an admissible $j$ of  type $3^-$ is of the form $j=d_1(2s+1)$ for some integer $s$, and that an admissible $j$ of type $2$ either $n/3$ or $2n/3$. As $n$ is divisible by $3$ and $d_1, d_2$ are relatively prime, we have that $3$ divides exactly one of $d_1, d_2$. Assume first that $3 | d_2$. Similarly, as in the proof of Proposition \ref{prop:type23+} we find that for every admissible character $\chi_j$ we have $\chi_j(2^t d_2)=1$, and so Proposition~\ref{prop:eigenspace} implies that $d_1=1$.  However, since $d_2$ is a divisor of $\ell_2$,
	$$
	c-1=2^{t-1}\ell_2< \frac{n}{2}=2^{t-1}d_2
	$$
	yields a contradiction.  
	
	Now, suppose that $3\mid d_1$. Similarly, as in the proof of Proposition \ref{prop:type23+} we see that for any admissible character $\chi_j$ we have that $\chi_{j}(2^t 3 d_2)=1$. Proposition~\ref{prop:eigenspace} implies that $2^t3d_2=kn=k2^td_1d_2$ for some $k\in\mathbb{Z}$. Therefore, $3=kd_1$, and since $3\mid d_1$ we must have $d_1=3$ and $k=1$. As $c-1=2^{t-1}\ell_2 < \frac{n}{2}=2^{t-1}3d_2$, we have that $\ell_2< 3d_2$. As $\ell_2$ is an odd multiple of $d_2$, this shows that $ \ell_2=d_2$. Summarizing all together, we have $n=2^t 3 d_2$, $c-1=2^{t-1} d_2$, and so $c+1=2^{t-1} d_2+2=2(2^{t-2} d_2+1)$ for some odd $d_2$, which is not divisible by $3$. But as $3|d_1$, we have that $3|c+1$, and so $3$ is a divisor of $2^{t-2} d_2 + 1$. If $t$ is even, then we have that $2^{t-2}\equiv1 \pmod{3}$, and so $d_2\equiv -1 \pmod{3}$. As $d_2$ is odd, $d_2=6k-1$ for some positive integer $k$. If $t$ is odd, then we have that $2^{t-2}\equiv-1 \pmod{3}$, and so $d_2\equiv 1 \pmod{3}$. As $d_2$ is odd, $d_2=6k+1$ for some nonnegative integer $k$. This shows that $\Ga$ is isomorphic to a graph belonging to the family described in part (iv) of Theorem \ref{thm:val5}.
\end{proof}

 \section{Proof of Theorem \ref{thm:val5} - A sufficient condition}
 In this section we prove that graphs described in the statement of  Theorem \ref{thm:val5} are indeed closed distance magic.
 
 Assume first that $\Ga=\Cay(\ZZ_n;\{\pm 1, \pm (n/2-1), n/2\})$ is a graph described in part (i) of  Theorem \ref{thm:val5}. Label the vertices of $\Ga$ as follows. For $0 \le x \le n/2-1$ let $\ell(x)=x+1$. For $n/2 \le x \le n-1$ let $\ell(x) = 3n/2 - x$. It is now easy to see that $\ell$ is a closed distance magic labeling of $\Ga$.
 
 Assume next that $\Ga=\Cay(\ZZ_n;\{\pm 1, \pm c, n/2\})$ is a graph described in part (ii) of  Theorem \ref{thm:val5}. Recall that in this case $n \equiv 2 \pmod{4}$, $c$ is even, and $2(c^2-1)$ is an odd multiple of $n$. By \cite[Theorem 1.1]{miklavivc2021classification}, graph $\Ga'=\Cay(\ZZ_n;\{\pm 1, \pm c\})$ is a tetravalent distance magic graph. Also note that distance magic labeling $\ell$ of the vertices of graph $\Ga'$, given in \cite[Proposition 4.2]{miklavivc2021classification}, has the property that for every $x \in \ZZ_n$, we have $\ell(x+n/2) = n+1-\ell(x)$. It follows that $\ell(x) + \ell(x+n/2) = n+1$ holds for every $x \in \ZZ_n$, and so $\ell$ is also a closed distance magic labeling of the vertices of graph $\Ga$.
 
 Assume next that $\Ga=\Cay(\ZZ_n;\{\pm 1, \pm c, n/2\})$ is a graph described in part (iii) of  Theorem \ref{thm:val5}. Recall that in this case $n=3 \cdot 2^t(6k+(-1)^t)$ and $c=2^{t-1}(6k+(-1)^t)-1$ for some integer $t \ge 2$ and some integer $k \ge 0$ such that $c \ge 2$. To define the labeling of vertices of $\Ga$, we split the vertex set of $\Ga$ into cosets of the subgroup $H$ of $\ZZ_n$, generated by $n/6$. For every $0 \le k \le n/6-1$, let $C_{3k}$ be the (ordered) coset defined by
 $$
   C_{3k} = 3k+H = \{3k, 3k+n/6, 3k+2n/6, 3k+3n/6, 3k+4n/6, 3k+5n/6\}.
 $$
 As $n$ is not divisible by $9$ and $0 \le k \le n/6-1$, the cosets are pairwise disjoint, and so every $x \in \ZZ_n$ belongs to exactly one of these cosets. Now define the labeling $\ell$ of the vertices of $\Ga$ as follows: for $0 \le k \le n/6-1$, let
 $$
   \ell(3k)= 1+3k, \qquad \ell(3k+2n/6)=3+3k, \qquad \ell(3k+4n/6)=2+3k,
 $$
 and
 $$
 \ell(3k+n/6)= n-1-3k, \quad \ell(3k+3n/6)=n-2-3k, \quad \ell(3k+5n/6)=n-3k.
 $$ 
 Observe that $\ell$ is a bijection from $\ZZ_n$ to $\{1,2, ..., n\}$. It remains to prove that for every $x \in \ZZ_n$ we have that
 \begin{equation}
 \label{eq:lab}
   \ell(x) + \ell(x+n/2) + \ell(x-1) + \ell(x+1) + \ell(x-c) + \ell(x+c) = 3(n+1).
 \end{equation}
 The key observation is that $x, x+n/2$ belong to the same coset of $H$, and since $c=n/6-1$, also $x+1,x-c$ and $x-1,x+c$ belong to the same coset of $H$. We consider the case $x=3k$ for some $0 \le k \le n/6-1$ in details; the other cases are treated similarly and are therefore left to the reader. 
 
 If $x=3k$ for some $0 \le k \le n/6-1$, then we have $\ell(x) + \ell(x+n/2) = 1+3k+n-2-3k=n-1$. Furthermore, $x-1 \equiv -1 \pmod{3}$ and $x+1 \equiv 1 \pmod 3$. Observe that $3n/6 \equiv 0 \pmod{3}$, and as $2^t \equiv (-1)^t \pmod{3}$ for every positive integer $t$, we also have
 $$
   {n \over 6} \equiv -1 \pmod{3}, \quad  {2n \over 6} \equiv 1 \pmod{3}, \quad  {4n \over 6} \equiv -1 \pmod{3},  \quad  {5n \over 6} \equiv 1 \pmod{3}.
 $$
 It follows from the above comments that $x-1=3k_1 + n/6$ or $x-1=3k_1+4n/6$ for some $0 \le k_1 \le n/6-1$, and that $x+1=3k_2 + 2n/6$ or $x+1 = 3k_2+5n/6$ for some $0 \le k_2 \le n/6-1$. This gives us four different cases that has to be considered. 
 
 If $x-1=3k_1+n/6$, then $x+c=3k_1+2n/6$, and so $\ell(x-1) + \ell(x+c) = n-1-3k_1+3+3k_1=n+2$.  If however $x-1=3k_1+4n/6$, then $x+c=3k_1+5n/6$, and so $\ell(x-1) + \ell(x+c) = 2+ 3k_1+n-3k_1=n+2$.  If $x+1=3k_2+2n/6$, then $x-c=3k_1+n/6$, and so $\ell(x+1) + \ell(x-c) = 3+3k_2+n-1-3k_2=n+2$.  If however $x+1=3k_2+5n/6$, then $x-c=3k_2+4n/6$, and so $\ell(x+1) + \ell(x-c) = n- 3k_2+2+3k_2=n+2$. Therefore, in all four cases we have that  \eqref{eq:lab} holds. This shows that $\ell$ is a closed distance magic labeling of the vertices of $\Ga$.
 
 \begin{figure}[t]
 	\centering
 	\includegraphics[width=7cm]{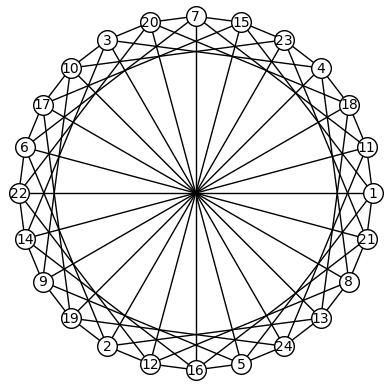}
 	\caption{Closed distance magic labeling for $\Cay(\ZZ_{24};\{\pm 1,\pm 5,12\})$.}
 	\label{fig11}
 \end{figure}
 
 Assume finally that $\Ga=\Cay(\ZZ_n;\{\pm 1, \pm c, n/2\})$ is a graph described in part (iv) of  Theorem \ref{thm:val5}. Then it turns out that the same labeling of the vertices of $\Ga$ as in the previous case is a closed distance magic labeling of the vertices of $\Ga$. The proof is similar and therefore we left the details to the reader. See also Figure \ref{fig11} for a closed distance magic labeling of the circulant that corresponds to $k=0, t=3$ in part (iv) of  Theorem \ref{thm:val5}, namely $\Cay(\ZZ_{24};\{\pm 1,\pm 5,12\})$.
 
 \subsection*{Acknowledgment}
 Blas Fernández's work is supported in part by the Slovenian Research Agency (research program P1-0285, research projects J1-2451, J1-3001 and J1-4008, and Young Researchers Grant). \v{S}tefko  Miklavi\v{c}'s  research is supported in part by the Slovenian Research Agency (research program P1-0285 and research projects J1-1695, N1-0140, N1-0159, J1-2451, N1-0208, J1-3001, J1-3003, J1-4008 and J1-4084).
 Roghayeh Maleki and Sarobidy Razafimahatratra's research are supported in part by the Ministry of Education, Science and Sport of Republic of Slovenia (University of Primorska Developmental funding pillar).
  

 \end{document}